\pdfoutput=1
\RequirePackage{ifpdf}
\ifpdf 
\documentclass[pdftex]{sigma}
\else
\documentclass{sigma}
\fi

\numberwithin{equation}{section}

\newtheorem{Theorem}{Theorem}[section]
 { \theoremstyle{definition}
\newtheorem{Remark}[Theorem]{Remark} }

\newcommand{\C}{{\mathbb C}}
\newcommand{\Z}{{\mathbb Z}}
\newcommand{\K}{{\mathbb K}}
\newcommand{\dcal}{{\mathcal D}}

\begin{document}
\allowdisplaybreaks

\newcommand{\arXivNumber}{1904.00272}

\renewcommand{\PaperNumber}{043}

\FirstPageHeading

\ShortArticleName{A Note on Spectral Triples on the Quantum Disk}

\ArticleName{A Note on Spectral Triples on the Quantum Disk}

\Author{Slawomir KLIMEK~$^\dag$, Matt MCBRIDE~$^\ddag$ and John Wilson PEOPLES~$^\dag$}

\AuthorNameForHeading{S.~Klimek, M.~McBride and J.W.~Peoples}

\Address{$^\dag$~Department of Mathematical Sciences, Indiana University-Purdue University Indianapolis,\\
\hphantom{$^\dag$}~402 N.~Blackford St., Indianapolis, IN 46202, USA}
\EmailD{\href{mailto:sklimek@math.iupui.edu}{sklimek@math.iupui.edu}, \href{mailto:wipeople@iu.edu}{wipeople@iu.edu}}

\Address{$^\ddag$~Department of Mathematics and Statistics, Mississippi State University,\\
\hphantom{$^\ddag$}~175 President's Cir., Mississippi State, MS 39762, USA}
\EmailD{\href{mailto:mmcbride@math.msstate.edu}{mmcbride@math.msstate.edu}}

\ArticleDates{Received April 03, 2019, in final form May 24, 2019; Published online May 28, 2019}

\Abstract{By modifying the ideas from our previous paper [\textit{SIGMA} \textbf{13} (2017), 075, 26~pages, arXiv:1705.04005], we construct spectral triples from implementations of covariant derivations on the quantum disk.}


\Classification{46L87; 46L89; 58B34; 58J42}

\section{Introduction}
Spectral triples are a key tool in noncommutative geometry \cite{Connes}, since they allow using analytical methods in studying quantum spaces.
In this note we show how, by changing the concept of a Hilbert space implementation of an unbounded derivation, one can use the techniques of our papers \cite{KMRSW,KMR} to construct meaningful, geometrical spectral triples for the quantum disk, the Toeplitz $C^*$-algebra of the unilateral shift.

Our previous paper \cite{KMRSW} classified unbounded derivations, covariant with respect to a natural rotation, and their implementations in Hilbert spaces obtained from the GNS construction with respect to invariant states. Surprisingly, no implementation of a covariant derivation in any GNS Hilbert space for a faithful, normal, invariant state turned out to have compact parametrices for a large class of boundary conditions. However, if we relax the concept of an implementation by allowing operators to act between different Hilbert spaces, then it turns out, as demonstrated in this paper, that there is an interesting class of examples of spectral triples that can be constructed this way.

In \cite{CKW,KM} our earlier attempts at constructing Dirac-type operators on the quantum disk using APS-type boundary conditions to eliminate infinite dimensional kernel/cokernel did not lead to examples of spectral triples. The class of operators considered in~\cite{CKW,KM}, designed to mimic the classical Atiyah--Patodi--Singer theory, does not have bounded commutators with representations of the algebra. Additionally, as pointed out in \cite{FMR}, there are fundamental reasons why APS boundary are not compatible with spectral triples even in classical geometry for algebras of functions which are non-constant on the boundary, as the corresponding domains of the Dirac-type operators are not preserved by the representations of the algebra. Fortunately, for similar yet quite different Dirac-type operators considered in this paper, coming from implementations of derivations, there is enough flexibility that the size of the kernel can be controlled by the growth conditions of the coefficients, as observed previously in \cite{KMRSW,KMR}. Such an option does not seem to have an obvious analog for the classical Dirac operator.

Other examples of spectral triples of the Toeplitz algebra, in GNS Hilbert spaces of non-normal states, were also constructed in \cite{C}, \cite[Section~4.2]{CM}, \cite{DD} and in \cite{EFI}. Those authors were working with irreducible representations of the Toeplitz algebra, unlike the representations considered in this paper.

We review the notation and basic concepts from~\cite{KMRSW} below and, in a number of places, we use the results contained in that reference.

\section{Quantum disk}
Let $\{E_k\}_{k=0}^\infty$ be the canonical basis for $\ell^2(\Z_{\geq 0})$ and $U$ be the unilateral shift defined by
\begin{gather*}UE_k = E_{k+1}.\end{gather*}
 Note that $U$ is an isometry, i.e., $U^*U = I$. Consider the Toeplitz algebra $A = C^*(U)$, the $C^*$-algebra generated by~$U$. This algebra is called the quantum disk. We also use the diagonal label operator
\begin{gather*}\K E_k = kE_k.\end{gather*}
It follows that for $a\colon \Z_{\geq 0} \to \C$, we have
\begin{gather*}a(\K) E_k = a(k) E_k.\end{gather*}
These are precisely the operators which are diagonal with respect to $\{E_k\}$. The operators $(\K,U)$ serve as noncommutative polar coordinates, and they satisfy
the following relation
\begin{gather*}\K U=U(\K+I).\end{gather*}

Let $c$ be the set of $a(k)$, as above, which are convergent as $k \to \infty$ and let $c_{00}^{+}$ be the set of all eventually constant functions, i.e., functions $a(k)$ such that there exists $k_0$ where $a(k)$ is constant for $k \geq k_0$.

Consider, for future reference, the following dense $*$-subalgebra of $A$
\begin{gather*}
\mathcal{A} = \bigg\{a = \sum_{n\geq 0}U^n a_n(\K)+\sum_{n< 0}a_n(\K)(U^*)^{-n} \colon a_n(k)\in c_{00}^+,\ \textrm{finite sums}\bigg\}.
\end{gather*}
By Proposition 3.1 in \cite{KMRSW}, $\mathcal{A} = \operatorname{Pol}(U, U^*)$.

\section{Derivations on quantum disk}
Let $\rho_{\theta}\colon A \to A $, $0\leq\theta<2\pi$, be a one parameter group of automorphisms of $A$ defined by
\begin{gather*}\rho_\theta (a) = {\rm e}^{{\rm i}\theta\K}a{\rm e}^{-{\rm i}\theta\K}.\end{gather*}
Since $\rho_\theta (U) = {\rm e}^{{\rm i}\theta}U$ and $\rho_\theta (U^*) = {\rm e}^{-{\rm i}\theta}U^*$, the automorphisms $\rho_\theta$ are well defined on~$A$
and they preserve $\mathcal{A}$. By Proposition~4.2 in~\cite{KMRSW}, any densly-defined derivation $d\colon \mathcal{A} \to A$, covariant with respect to $\rho_\theta$
\begin{gather*}\rho_\theta(d(a)) = {\rm e}^{{\rm i}\theta}d(\rho_\theta(a)),\end{gather*}
is of the following form
\begin{gather*}d(a) = [U\beta(\K),a],\end{gather*}
where $\beta(k+1) - \beta(k) \in c$. We use notation
\begin{gather*}
\lim_{k \to \infty}\beta(k+1) - \beta(k) := \beta_\infty,
\end{gather*}
and below we only consider covariant derivations with $\beta_\infty\ne 0$.

\section{Covariant implementations on quantum disk}
We will begin by introducing the following family of states $\tau_w\colon A \to \C$ on $A$, defined by
\begin{gather*}\tau_w (a) = \operatorname{tr}(w(\K)a),\end{gather*}
where $w(k) > 0$ for all $k \in \Z_{\geq 0}$ and
\begin{gather*}\sum_{k=0}^{\infty}{w(k)}=1.\end{gather*}

As a result of Proposition~5.4 in \cite{KMRSW}, $\tau_w$ are precisely the $\rho_\theta$-invariant, normal, faithful states on~$A$. Let $H_w$ be the Hilbert space obtained by Gelfand--Naimark--Segal (GNS) construction on~$A$ using state~$\tau_w$. Since the state is faithful, $H_w$ is the completion of~$A$ with respect to the inner product given by
\begin{gather*}(a,b)_w = \tau(w(\K)a^*b).\end{gather*}
A simple calculation leads to the following precise description:
$H_w$ is the Hilbert space consisting of infinite series of operators
\begin{gather}\label{f_expansion}
f = \sum_{n\geq 0}U^n f_n(\K)+\sum_{n< 0}f_n(\K)(U^*)^{-n}
\end{gather}
satisfying
\begin{gather*}
\|f\|_w^2 = \sum_{n \geq 0} \sum_{k \geq 0} w(k)|f_n(k)|^2 + \sum_{n < 0} \sum_{k \geq 0} w(k-n)|f_n(k)|^2 < \infty.
\end{gather*}
It is important to notice that $\mathcal{A} \subseteq H_w$ and that $\mathcal{A}$ is dense in $H_w$. The space of all formal series~\eqref{f_expansion} with arbitrary coefficients~$f_n(k)$ will be denoted by $\mathcal F$.

The GNS representation map $\pi_w\colon A \to B(H_w)$ is given by left-hand multiplication
\begin{gather*}\pi_w (a)f = af.\end{gather*}

Define a one parameter group of unitary operators $U_\theta ^w\colon H_w \to H_w$ via the formula
\begin{gather*}U_\theta^wf = \sum\limits_{n \geq 0}U^n{\rm e}^{{\rm i}n\theta}f_n(\K) + \sum\limits_{n<0}{\rm e}^{{\rm i}n\theta}f_n(\K)(U^*)^{-n}.\end{gather*}
It is easily seen that they are implementing $\rho_\theta$, as we have
\begin{gather*}
 \pi_w(\rho_\theta(a)) = U_\theta^w\pi_w(a)\big(U_\theta^w\big)^{-1}.
\end{gather*}

Consider an additional weight, $w'(k)$, possibly different from~$w(k)$, satisfying the same conditions. An operator $D\colon H_w\supseteq \mathcal{A} \to H_{w'}$ is called a covariant implementation of a~covariant derivation $d$ if for every $a\in\mathcal{A}$, and for every $f\in\mathcal{A}$ considered as an element of both~$H_w$ and~$H_{w'}$, we have
\begin{gather*}D\pi_w(a)f -\pi_{w'}(a)Df = \pi_{w'}(d(a))f,\end{gather*}
and, additionally, $D$ satisfies
\begin{gather*}U_\theta^{w'} D \big(U_\theta^w\big)^{-1}f={\rm e}^{{\rm i}\theta}Df.
\end{gather*}
Allowing for implementations between different Hilbert spaces is the key difference between this paper and~\cite{KMRSW}.

Exactly the same argument as in Proposition~6.1 in~\cite{KMRSW} shows that any implementation $D$ is of the form
\begin{gather}\label{D_formula}
Df = U\beta(\K)f - fU\alpha(\K),
\end{gather}
where $\alpha(k)$ is a sequence such that
\begin{gather} \label{one}
\sum_{k=0}^\infty|\beta(k) - \alpha(k)|^2w'(k) < \infty.
\end{gather}

The assumption $\beta_\infty\ne 0$ implies that $\beta(k)$ has at most finitely many zeros. Without loss of generality, we may assume that $\beta(k) \neq 0$ as this can be obtained by a bounded perturbation. Arguments in \cite{KMRSW} show that if $\alpha(k)$ has infinitely many zeros, then $\operatorname{Ker}(D)$ has infinite dimension, and so $D$ cannot define a spectral triple. Hence we assume that for every $k$ we have
\begin{gather} \label{three}
 \alpha(k), \beta(k)\neq 0.
\end{gather}

Additionally, as in \cite{KMRSW}, we write
 \begin{gather} \label{four}
 \alpha(k) = \beta(k)\frac{\mu(k+1)}{\mu(k)},\qquad \text{where} \quad \mu(0) = 1.
 \end{gather}
Notice that the operator $D$, given by the formula~\eqref{D_formula}, is a well-defined linear map on $\mathcal{F}$ and $D\colon \mathcal F\to \mathcal F$, where again~$\mathcal{F}$ is the space of all formal series defined by equation~\eqref{f_expansion}. This allows us to describe the maximal domain of $D$ in $H_w$ as
\begin{gather*}
\textrm{dom}(D) = \left\{ f \in H_w\subseteq\mathcal{F}\colon Df \in H_{w'} \right\}.
\end{gather*}

Arguing exactly as in the paragraph preceding Proposition~5 of~\cite{CKW}, at least for the choices of parameters at the end of the paper, the operator $D$ has a natural radial/angular decomposition in full analogy with the classical $d$-bar operator on the disk.

The following is the main technical result of this paper.

\begin{Theorem}\label{compact_para}Assume the following
\begin{gather} \label{five}
\left|\frac{\beta(k) \cdots \beta(k+n)}{\beta(j) \cdots \beta(j+n)}\right| \leq {\rm const} \qquad \textrm{for all} \quad k \leq j, \quad n\in\Z_{\geq0},
\\ \label{six}
\sum_{k=0}^{\infty}\sum_{j=0}^{\infty}\frac{1}{(\max (j,k) + 1)^2}\left|\frac{\mu(j)}{\mu(k)}\right|^2\frac{w(k)}{w'(j)} < \infty,
\end{gather}
there exists $N$ such that
\begin{gather}
\label{seven}
\sum_{k=0}^{\infty} \frac{(1+k)^{2n}}{|\mu(k)|^{2}}w(k) <\infty
\end{gather}
for $n<N$ and infinite for $n\geq N$.
Assume additionally that $\beta_\infty\ne 0$, and \eqref{one}, \eqref{three}, and \eqref{four} hold. Then $D$ on $\operatorname{dom}(D)$ has a compact parametrix.
In fact, there exists a compact operator $Q\colon H_{w'} \to H_w$, satisfying $DQ = I$ and $QD = I - C$ where C is a compact operator.
\end{Theorem}

\begin{proof}We begin by expressing the action of $D$ in terms of the Fourier decomposition
\begin{gather*}
 Df = \sum_{n\geq 0}U^{n+1}(D_n f_n)(\K)+\sum_{n< 0}(D_n f_n)(\K)(U^*)^{-n-1},
\end{gather*}
where $D_n$ are given by the following formulas:
for $n \geq 0$ and $f\in \textrm{dom}(D_n) = \big\{ f \in \ell^2_{w}\colon D_nf \in \ell^2_{w'} \big\}$,
\begin{gather*}
D_nf(k) = \beta(k+n)f(k) - \beta(k)\frac{\mu(k+1)}{\mu(k)}f(k+1),
\end{gather*}
and for $n<0$ and $f \in \operatorname{dom}(D_n)= \big\{ f \in \ell^2_{w_n}\colon D_nf \in \ell^2_{w'_{n+1}} \big\}$,
\begin{gather*}
D_nf(k) = \beta(k-n-1)\frac{\mu(k-n)}{\mu(k-n-1)}f(k) - \beta(k-1)f(k-1),
\end{gather*}
where we used the following notation
\begin{gather*}\ell^2_w=\left\{f(k)\colon \sum_{k=0}^\infty|f(k)|^2w(k)<\infty\right\},
\end{gather*}
and
\begin{gather*}\ell^2_{w_n}=\left\{f(k)\colon \sum_{k=0}^\infty|f(k)|^2w(k-n)<\infty\right\}.
\end{gather*}
Naturally, two cases ($n\geq 0$ and $n<0$) arise.

Case 1 ($n\geq 0$): We begin by looking at $n\geq N$ from condition~(\ref{seven}). The formal kernel of~$D_n$ is one-dimensional and spanned by the following vector
\begin{gather*}
h^{(n)}(k) = \frac{\beta(k) \cdots \beta(k+n-1)}{\mu(k)}.
\end{gather*}
By formula (\ref{seven}), we have $\big\|h^{(n)}\big\|_{\ell^2_w} = \infty$, so the kernel of $D_n$ in $H_w$ is trivial.

Consider the following operator $Q_n\colon \ell^2_{w'}\to \ell^2_{w}$
\begin{gather*}
 Q_ng = \sum_{j=k}^{\infty} \frac{\beta(k)\cdots\beta(k+n-1)}{\beta(j)\cdots\beta(j+n)}\frac{\mu(j)}{\mu(k)}g(j).
\end{gather*}
We have the following formula for the Hilbert--Schmidt norm of $Q_n$
\begin{gather*}
 \|Q_n\|_{\rm HS}^2 = \sum_{k=0}^\infty \sum_{j=k}^\infty \left| \frac{\beta(k)\cdots\beta(k+n-1)}{\beta(j)\cdots\beta(j+n)}\right|^2\left|\frac{\mu(j)}{\mu(k)}\right|^2\frac{w(k)}{w'(j)}.
\end{gather*}
Therefore, using assumption (\ref{five}), we can estimate as follows
\begin{gather*}
 \|Q_n\|_{\rm HS}^2 \leq \sum_{k=0}^\infty \sum_{j=k}^\infty \left|\frac{\beta(k)\cdots\beta(k+n-1)}{\beta(j)\cdots\beta(j+n-1)}\right|^2\frac{1}{|\beta(j+n)|^2}\left|\frac{\mu(j)}{\mu(k)}\right|^2\frac{w(k)}{w'(j)} \leq \\
\hphantom{\|Q_n\|_{\rm HS}^2}{} \leq \textrm{const}\sum_{k=0}^\infty \sum_{j=k}^\infty \frac{1}{(1+j+n)^2}\left|\frac{\mu(j)}{\mu(k)}\right|^2\frac{w(k)}{w'(j)}.\\
\end{gather*}
Consequently, we have $\|Q_n\|_{\rm HS}^2 < \infty$ by (\ref{six}) and $\|Q_n\|_{\rm HS} \to 0$ as $n \to \infty$ by Lebesgue's do\-mi\-nated convergence theorem.

An easy calculation shows that $D_nQ_ng=g$ for every $g\in \ell^2_{w'}$. Consequently we have: $\operatorname{Ran}(Q_n)\subseteq \operatorname{dom}(D_n)$ and moreover $\operatorname{Ran}(D_n)=\ell^2_{w'}$. Given any $f\in \operatorname{dom}(D_n)$ the difference $Q_nD_nf-f$ is in the domain of~$D_n$, and
\begin{gather*}D_n(Q_nD_nf-f)=0.\end{gather*}
But the kernel of $D_n$ is trivial so $Q_nD_nf=f$, and by continuity $Q_nD_n=I$, therefore $Q_n$ is the inverse of~$D_n$.

\begin{Remark}Notice that in the $n\geq N$ case the maximal domain $\operatorname{dom}(D_n)$ is equal the minimal domain $\operatorname{dom}_{\min}(D_n)$, defined as the closure of $c_{00}$ with respect to the graph norm. Here $c_{00}$ is the set of sequences that are eventually zero. Given any $f\in \operatorname{dom} (D_n)$ we have that $D_nf\in\ell^2_{w'}$ and we choose a sequence $\big\{\tilde f_j\big\}$, $\tilde f_j\in c_{00}$ converging to $D_nf$. Consider the sequence $f_j=Q_n\tilde f_j$, which is in $c_{00}$ as easily seen from the formula for~$Q_n$. Then we have $f_j\to f$ by continuity of~$Q_n$ and $D_nf_j\to D_nf$, so that $f$ in the closure of~$c_{00}$.
\end{Remark}

We will now consider $0 \leq n < N$. We define the following operator
\begin{gather*}
 \tilde{Q}_ng(k) = \sum_{j=k}^{\infty} \frac{\beta(k) \cdots \beta(k+n-1)}{\beta(j) \cdots \beta(j+n)} \frac{\mu(j)}{\mu(k)}g(j).
\end{gather*}
By the results in the case $n\geq N$, previously discussed, $\tilde{Q}_n$ is a Hilbert--Schmidt operator from~$\ell^2_{w'}$ to~$\ell^2_{w}$. We will verify that the operator
\begin{gather*}Q_ng(k) = \tilde{Q}_ng(k) - \frac{\tilde{Q}_ng(0)}{\beta(0)\cdots\beta(n-1)}h^{(n)}(k)\end{gather*}
is a parametrix of $D_n$. The second term in the above expression is a rank~1 operator, so $Q_n$ is still a Hilbert--Schmidt operator. Additionally
\begin{gather*}
 D_nQ_ng = D_n\tilde{Q}_ng -\frac{\tilde{Q}_ng(0)}{\beta(0)\cdots\beta(n-1)}D_n\big(h^{(n)}\big) = g
\end{gather*}
by results in the previous case.

Given $f\in$ dom$(D_n)$ we have
\begin{gather*}
Q_nD_nf(k)= \lim_{L\to\infty} \sum_{j=k}^{L}\frac{\beta(k) \cdots \beta(k+n-1)}{\beta(j) \cdots \beta(j+n)} \frac{\mu(j)}{\mu(k)} \left( \beta(j+n)f(j) - \beta(j)\frac{\mu(j+1)}{\mu(j)}f(j+1) \right)\\
 \hphantom{Q_nD_nf(k)=}{} -\sum_{j=0}^{L} \frac{\mu(j)}{\beta(j) \cdots \beta(j+n)} \left(\beta(j+n)f(j)-\beta(j) \frac{\mu(j+1)}{\mu(j)}f(j+1)\right)h^{(n)}(k).
\end{gather*}
Simplifying the telescoping sum yields
\begin{gather*}
 Q_nD_nf(k) = f(k) - \frac{f(0)}{\beta(0) \cdots \beta(n-1)}h^{(n)}(k).
\end{gather*}
Clearly, by the imposed conditions,{\samepage
\begin{gather*}C_nf(k) := \frac{f(0)}{\beta(0) \cdots \beta(n-1)}h^{(n)}(k)\end{gather*} is a
Hilbert--Schmidt operator, and $Q_nD_n=I-C_n$, proving that $Q_n$ is a parametrix of $D_n$.}

Case 2 ($n < 0$): Clearly $\operatorname{Ker}(D_n) = 0$ and $D_n$ is invertible with inverse
\begin{gather*}
Q_ng(k) = \sum_{j=0}^k\frac{\beta(j) \cdots \beta(j-n-2)}{\beta(k) \cdots \beta(k-n-1)}\frac{\mu(j-n-1)}{\mu(k-n)}g(j).
\end{gather*}
The above is done by direct calculation and can be verified by checking $Q_nD_n = I$ and $Q_nD_n = I$,
see additionally Lemma~7.7 in~\cite{KMRSW}. We will now show that $Q_n$ is a Hilbert--Schmidt operator. Consider the Hilbert--Schmidt norm
\begin{gather*}
 \|Q_n\|_{\rm HS}^2 = \sum_{k=0}^{\infty}\sum_{j=0}^k \left|\frac{\beta(j) \cdots \beta(j-n-2)}{\beta(k) \cdots \beta(k-n-1)}\right|^2\left|\frac{\mu(j-n-1)}{\mu(k-n)}\right|^2\frac{w(k-n)}{w'(j-n-1)}.
\end{gather*}
Changing the indices $k-n \to k'$ and $j-n-1 \to j'$, the order of summation, and estimating as above yields
\begin{gather*}
 \|Q_n\|_{\rm HS}^2 \leq \sum_{j=-n-1}^\infty\sum_{k=j+1}^\infty \frac{1}{(1+k)^2}\left|\frac{\mu(j)}{\mu(k)}\right|^2 \frac{w(k)}{w'(j)} < \infty.
\end{gather*}
It follows that $\|Q_n\|_{\rm HS}^2 \to 0$ as $-n \to \infty$ since this is the tail of an absolutely convergent series.

Consequently, in all cases, $D_n$ has Hilbert--Schmidt, and thus compact, parametrices $Q_n$, with Hilbert--Schmidt norms approaching zero. Thus, it follows that~$D$ has a compact parametrix. This completes the theorem. \end{proof}

The main significance of this result is outlined in the following theorem. First, we introduce some notation related to spectral triples as considered in~\cite{KMRSW}.

Let $H = H_{w'} \bigoplus H_w$, with grading $\Gamma \big|_{H_{w'}} =1$ and $\Gamma \big|_{H_w} = -1$. Define a representation $\pi\colon A \to B(H)$ of $A$ in $H$ by the formula
\begin{gather*}\pi(a) = (\pi_{w'}(a),\pi_w(a)),\end{gather*}
and also define a quantum analog of a Dirac operator on the unit disk by
\begin{gather*}
\dcal = \left[
\begin{matrix}
0 & D \\
D^* & 0
\end{matrix}\right],
\end{gather*}
so that $\pi(a)$ are even and $\dcal$ is odd with respect to grading $\Gamma$.

\begin{Theorem}With the above notation, $(\mathcal{A},H,\dcal)$ forms an even spectral triple over~$A$.
\end{Theorem}
\begin{proof}By Theorem \ref{compact_para}, we have that $D$ is has a compact parametrix and so $\dcal$ has compact parametrices by the results in the appendix of~\cite{KMRSW}. Clearly~$\pi(a)$ preserves the domain of~$D$ and, since~$D$ is an implementation of a derivation $d\colon \mathcal{A} \to A$, the commutator $[\dcal,\pi(a)]$ is bounded for all $a\in\mathcal{A}$. This completes the proof.
\end{proof}

We conclude this paper by giving explicit examples of parameters $w(k)$, $w'(k)$, $\beta(k)$, and~$\mu(k)$ that satisfy the conditions of the above theorems.

Assume $3<a<2b-1<c$, nonnegative $N\geq\frac{c-2b-1}{2}$, $\beta(k) = 1+k$, and consider the following sequences:
\begin{gather*}
w'(k) = \frac{w'(0)}{(1+k)^a} , \qquad \mu(k) = \frac{1}{(1+k)^b} ,\qquad w(k) = \frac{w(0)}{(1+k)^c},
\end{gather*}
where $w'(0)$ and $w(0)$ are such that $\sum\limits_{k=0}^\infty w'(k) = 1 = \sum\limits_{k=0}^\infty w(k).$ Then a straightforward calculation shows that they
satisfy the necessary conditions for $(\mathcal{A},H,\dcal)$ to be an even spectral triple. Moreover, there is a choice of parameters $a$, $b$, $c$ such that $N$ can be equal to zero.

\subsection*{Acknowledgments} We would like to thank an anonymous referee for pointing out a critical issue with the initial version of this paper.

\pdfbookmark[1]{References}{ref}
\LastPageEnding

\end{document}